\documentclass[11pt]{amsart}
\usepackage{amsmath,amssymb,amsthm}
\usepackage{geometry}
\newtheorem{theorem}{Theorem}
\newtheorem{lemma}[theorem]{Lemma}
\newtheorem{definition}[theorem]{Definition}
\newtheorem{remark}[theorem]{Remark}

\newtheorem{cor}[theorem]{Corollary}

\newtheorem*{coveringcriterion}{Two-sided covering criterion}
\newtheorem*{extensionlemma}{Uniform extension lemma}

\DeclareMathOperator{\RUC}{RUC}
\DeclareMathOperator{\Tame}{Tame}

\DeclareMathOperator{\Homeo}{Homeo}


\title{Tame Topological Groups are precompact}
\author{Eli Glasner}
\address{School of Mathematics\\
Tel Aviv University\\
Tel Aviv, Israel}
\email{glasner@math.tau.ac.il}

\date{August 12, 2026}

\begin{document}

\begin{abstract}
We prove that a Hausdorff topological group is tame if and only if it is precompact. 
\end{abstract}

\maketitle

\section{Introduction}

Let $G$ be a Hausdorff topological group.  The right uniform structure on
$G$ is generated by the entourages
\[
 R_U=\{(g,h)\in G\times G:gh^{-1}\in U\},
\]
where $U$ ranges over the open neighborhoods of the identity $e$.
Throughout, all Banach spaces and function spaces are over \(\mathbb R\).

\begin{definition}
For a compact space $X$, a {\em $G$-dynamical system}  $(X,G)$ is defined 
by a continuous homomorphism of $G$ into the group $\Homeo(X)$ of self homeomorphisms of $X$
equipped with the topology of uniform convergence. We denote by $gx$ the image of $x \in X$ under the
homeomorphism corresponding to $g \in G$ under this homomorphism.
\end{definition}

\begin{definition}
A bounded continuous function $f:G\to\mathbb R$ is \emph{right uniformly
continuous} if for every $\varepsilon>0$ there exists a neighborhood $U$ of
$e$ such that
\[
 |f(g)-f(h)|<\varepsilon
 \quad\text{whenever}\quad gh^{-1}\in U.
\]
The unital commutative Banach algebra of all such functions is denoted by
$\RUC(G)$.
\end{definition}

For \(g\in G\) and \(f\in\RUC(G)\), define
\[
 (f\cdot g)(h)=f(gh)
 \qquad (h\in G).
\]
Then
\[
 (f\cdot g)\cdot g'=f\cdot(gg')
 \qquad (g,g'\in G),
\]
so this defines a right action of \(G\) on \(\RUC(G)\). We denote the
orbit of \(f\) by
\[
 fG=\{f\cdot g:g\in G\}.
\]

\begin{definition}[\cite{Rosenthal1974}]
Let $(f_n)_{n\in\mathbb{N}}$ be a sequence of bounded real-valued functions
on  a set $X$.
\begin{enumerate}
\item
 It is called \emph{independent} if there exist real numbers
\(a<b\) such that, for every pair of finite disjoint sets
\(P,Q\subset\mathbb{N}\), the set
\[
\bigcap_{n\in P}\{x\in X:f_n(x)<a\}
\cap
\bigcap_{n\in Q}\{x\in X:f_n(x)>b\}
\]
is nonempty.
\item
It is an  {\em $\ell_1$ sequence} if there are strictly positive constants
$a$ and $b$ such that
$$
a\sum_{k=1}^n |c_k| \le \left\| \sum_{k=1}^n c_k f_k \right\|
\le a\sum_{k=1}^n |c_k|
$$
for all $n\in \mathbb{N}$ and $c_1,\dots,c_n\in \mathbb{R}$.
\end{enumerate}
\end{definition}

\begin{definition}
A function $f\in\RUC(G)$ is \emph{tame} if its  orbit
\[
f G=\{f \cdot g : g\in G\}
\]
does not contain a sequence equivalent to the standard basis of $\ell_1$.
A compact dynamical system $(X,G)$ is called tame if for every $F \in C(X)$ and $x \in X$, the 
RUC function $f(g) = F(gx)$ is tame. It follows that the collection $\Tame(G)$ consisting of tame functions
is a closed subalgebra of $RUC(G)$. 
Thus, the group $G$ is tame iff every point-transitive dynamical system $(X,G)$ is tame.
\end{definition}

\begin{remark}
In the last section we will show
 that, equivalently one can require that $fG$ does not contain an infinite independent sequence,
see e.g. \cite{Rosenthal1974, vD-89, Ko-95}. 
For more details on tame flows see e.g. \cite{GM-18}.
\end{remark}

\begin{definition}
A topological group $G$ is \emph{precompact} if for every open neighborhood
$U$ of $e$ there exists a finite set $F\subseteq G$ such that $G=UF$.
Equivalently, one may require $G=FU$.
\end{definition}

\begin{definition}
Let $V$ be a symmetric neighborhood of $e$.  A subset $D\subseteq G$ is
called \emph{$V$-discrete} if the sets $Vd$, $d\in D$, are pairwise
disjoint.  A set is \emph{right uniformly discrete} if it is $V$-discrete
for some symmetric identity neighborhood $V$.
\end{definition}

We shall use the following characterization of precompactness.

\begin{coveringcriterion}\label{crit:two-sided}
A topological group $G$ is precompact if and only if, for every identity
neighborhood $U$, there exist finite sets $A,B\subseteq G$ such that
\[
 G=AUB.
\]
\end{coveringcriterion}

\begin{proof}
If $G$ is precompact, then $G=UF$ for some finite $F$, so one may take
$A=\{e\}$ and $B=F$.

Conversely, suppose the displayed condition holds.  For a given $U$, choose
finite $A,B$ with $G=AUB$ and put $F=A\cup B\cup\{e\}$.  Then
\[
 G=AUB\subseteq FUF\subseteq G,
\]
so $G=FUF$.  The implication from this latter condition to precompactness is
Proposition~4.3 of Bouziad and Troallic \cite{BouziadTroallic2007} (see also \cite{Us-02}).
\end{proof}

We shall also use a standard extension fact for uniform spaces.  It is
included to avoid any metrizability assumption on $G$.

\begin{extensionlemma}\label{lem:extension}
Let $X$ be a uniform space, let $D\subseteq X$ be uniformly discrete, and
let $u:D\to\mathbb R$ be bounded.  Then $u$ extends to a bounded uniformly
continuous function $\widetilde u:X\to\mathbb R$ satisfying
\[
 \|\widetilde u\|_\infty=\|u\|_\infty.
\]
\end{extensionlemma}

\begin{proof}
Let $M=\|u\|_\infty$.  The assertion is immediate if $M=0$.  Choose an
entourage $E$ such that the sets $E[d]$, $d\in D$, are pairwise disjoint.
By the pseudometrization lemma for uniform spaces, there exist a uniformly
continuous pseudometric $p$ on $X$ and a number $\delta>0$ such that
\[
 \{(x,y):p(x,y)<\delta\}\subseteq E;
\]
see, for example, \cite[Chapter~I]{Isbell1964}.  Hence
$p(d,d')\geq\delta$ whenever $d,d'\in D$ are distinct.  It follows that
$u$ is $L$-Lipschitz on $(D,p)$ for $L=2M/\delta$.

The McShane formula (commonly known as the McShane--Whitney extension theorem)
\[
 U(x)=\inf_{d\in D}\bigl(u(d)+L p(x,d)\bigr),
\]
defines an $L$-Lipschitz extension of $u$ to $X$  
(see \cite{McShane1934}; the proof works, verbatim, for a pseudometric as well). 
Finally, set
\[
 \widetilde u(x)=\max\{-M,\min\{M,U(x)\}\}.
\]
Then $\widetilde u$ is uniformly continuous, agrees with $u$ on $D$, and
has norm $M$.
\end{proof}

\section{Main result}

It is well known that  every bounded RUC weakly almost periodic function is tame. 
Thus, $WAP(G)$, the subalgebra of 
$RUC(G)$ containing all the waekly almost periodic functions, is a subalgebra of $\Tame(G)$, see \cite{GM-18}.
It was shown in \cite{MPU-01} that every WAP topological group is precompact.

\begin{theorem}\label{thm:main}
Let $G$ be a Hausdorff topological group.  Then $G$ is tame if and only if
$G$ is precompact.
\end{theorem}

\begin{proof}
Suppose first that $G$ is precompact.  Its two-sided uniformity completion $K$ is a compact
topological group in which $G$ is dense.  Every $f\in\RUC(G)$ extends
uniquely to a function $\widehat f\in C(K)$.  The map
\[
 K\longrightarrow C(K),\qquad k\longmapsto L_k\widehat f,
\]
is continuous in the uniform norm.  Consequently, the orbit of
$\widehat f$ is norm compact, and the orbit of $f$ is norm precompact.
Thus every member of $\RUC(G)$ is almost periodic.  Since every almost
periodic function is tame, $G$ is tame.

Conversely, suppose that $G$ is not precompact.  By the two-sided covering
criterion, there exists an open identity neighborhood $U$ such that
\[
 G\neq AUB                                                   \tag{1}
\]
for every pair of finite sets $A,B\subseteq G$.  Choose a symmetric open
identity neighborhood $V$ satisfying $V^2\subseteq U$.  It follows from
(1) that
\[
 G\neq AV^2B                                                 \tag{2}
\]
for every pair of finite sets $A,B\subseteq G$.

We first construct a sequence $(g_n)_{n\geq1}$ such that 
\[
 g_ng_i^{-1}\notin V^2 \qquad (n\neq i).                    \tag{3}
\]
Let $g_1 =e$ and having chosen $g_1,\dots,g_{n-1}$, use (2) with
$A=\{e\}$ and $B=\{g_1,\dots,g_{n-1}\}$ to choose
\[
 g_n\notin V^2\{g_1,\dots,g_{n-1}\}.
\]
Since $V^2$ is symmetric, this gives (3) in both orders.

Let
\[
 \mathcal S=\bigcup_{N\geq1}\bigl(\{N\}\times\{-1,1\}^N\bigr),
\]
and enumerate this countable set as
\[
 (N_k,\varepsilon^{(k)}) \qquad (k\geq1),
\]
where
\[
 \varepsilon^{(k)}=
 \bigl(\varepsilon^{(k)}_1,\dots,\varepsilon^{(k)}_{N_k}\bigr)
 \in\{-1,1\}^{N_k}.
\]
Thus every finite sign pattern occurs in the enumeration.

\vspace{.3cm}

We shall choose the elements \(h_k \in G\) inductively so that, for every
\(k \geq 1\), the finite set
\[
D_k
=
\{g_i h_\ell : 1 \leq \ell \leq k,\ 1 \leq i \leq N_\ell\}
\]
is \(V\)-discrete; that is, the sets
\[
Vg_i h_\ell
\qquad
(1 \leq \ell \leq k,\ 1 \leq i \leq N_\ell)
\]
are pairwise disjoint. Equivalently, since \(V\) is symmetric, we require
\[
(g_i h_\ell)(g_j h_m)^{-1} \notin V^2
\]
whenever \((\ell,i) \neq (m,j)\).

Suppose that \(h_1,\ldots,h_{k-1}\) have already been chosen so that
\(D_{k-1}\) is \(V\)-discrete, where
\[
D_{k-1}
=
\{g_i h_\ell : 1 \leq \ell < k,\ 1 \leq i \leq N_\ell\}.
\]
We shall choose \(h_k\) so that the new block
\[
\{g_i h_k : 1 \leq i \leq N_k\}
\]
is \(V\)-separated both internally and from \(D_{k-1}\).

Define the finite sets
\[
 A_k=\{g_i^{-1}:1\leq i\leq N_k\},
 \qquad
 B_k=D_{k-1}\cup\{e\}.
\]
By (2), choose
\[
 h_k\notin A_kV^2B_k.                                      \tag{4}
\]
We claim that
\[
 D=\{g_i h_k:k\geq1,\ 1\leq i\leq N_k\}                  \tag{5}
\]
is $V$-discrete.

Indeed, consider two distinct points in the same $k$-th block. 
If $i\neq j$ and $ Vg_i h_k\cap Vg_j h_k\neq\varnothing$,
then $g_i g_j^{-1}\in V^2$, contrary to (3).  

Next, let
$y\in D_{k-1}$.  
If $Vg_i h_k\cap Vy\neq\varnothing$,
then $ g_i h_k y^{-1}\in V^2$, and therefore
$ h_k\in g_i^{-1}V^2y\subseteq A_kV^2B_k$,
contrary to (4).  Induction on $k$ proves the claim.

\vspace{.3cm}

Define $u:D\to\{-1,1\}$ by
\[
 u(g_i h_k)=\varepsilon^{(k)}_i
 \qquad (k\geq1,\ 1\leq i\leq N_k).
\]
Since $D$ is right uniformly discrete, the uniform extension lemma gives
an extension $f\in\RUC(G)$ with
$ \|f\|_\infty=1$.

We show that $(L_{g_n}f)_{n\geq1}$ is isometrically equivalent to the
standard basis of $\ell_1$.  Let $N\geq1$ and let
$a_1,\dots,a_N$ be real scalars.  Choose signs
$\varepsilon_i\in\{-1,1\}$ so that
\[
 a_i\varepsilon_i=|a_i| \qquad (1\leq i\leq N).
\]
There is some $k$ with
\[
 N_k=N
 \quad\text{and}\quad
 \varepsilon^{(k)}=(\varepsilon_1,\dots,\varepsilon_N).
\]
Evaluating at \(h_k\), we obtain
\begin{align*}
 \left(\sum_{i=1}^N a_i(f\cdot g_i)\right)(h_k)  & =\sum_{i=1}^N a_i(f\cdot g_i)(h_k)\\
 &=\sum_{i=1}^N a_i f(g_i h_k)  =\sum_{i=1}^N a_i\varepsilon_i\\
 &=\sum_{i=1}^N|a_i|.
\end{align*}
Hence
\[
 \left\|\sum_{i=1}^N a_i(f\cdot g_i)\right\|_\infty
 \geq\sum_{i=1}^N|a_i|.
\]
The reverse inequality follows from \(\|f\|_\infty=1\), and therefore
\[
 \left\|\sum_{i=1}^N a_i(f\cdot g_i)\right\|_\infty
 =\sum_{i=1}^N|a_i|.
\]
Thus the orbit \(fG\) contains a sequence isometric to the standard
basis of \(\ell_1\), so \(f\) is not tame. Consequently, a
non-precompact Hausdorff topological group is not tame.
\end{proof}

\section{Independence vs. $\ell_1$}

In this section we compare for a 
norm-bounded family of functions in $\ell_1$ the property of having an independent subsequence
and having an $\ell_1$ subsequence, see \cite{GM-18}.
It was shown by Rosenthal \cite[Proposition 4]{Rosenthal1974} that every bounded independent sequence is an \(\ell_{1}\)-sequence.
The other direction was shown to me by Michael Megrelishvili.

\begin{lemma}\label{lem:l1-independent-arbitrary-set}
Let \(X\) be an arbitrary set, and let
\[
\mathcal{F}\subset \ell_{\infty}(X)
\]
be a norm-bounded family. Then the following conditions are equivalent:
\begin{enumerate}
    \item \(\mathcal{F}\) contains an \(\ell_{1}\)-sequence;
    \item \(\mathcal{F}\) contains an infinite independent sequence.
\end{enumerate}
More precisely, every independent sequence in \(\ell_{\infty}(X)\) is an
\(\ell_{1}\)-sequence, and every \(\ell_{1}\)-sequence in
\(\ell_{\infty}(X)\) has an independent subsequence.
\end{lemma}

\begin{proof}
We first show that every independent sequence is an
\(\ell_{1}\)-sequence. 
This is  \cite[Prop. 4]{Rosenthal1974}. We include a proof for completeness. 

Let \((f_n)\) be independent, witnessed by
\(a<b\). Fix finitely many scalars \(c_1,\ldots,c_m\), and put
\[
P=\{i:c_i\geq 0\},
\qquad
Q=\{i:c_i<0\}.
\]
By independence, there exist \(x_{+},x_{-}\in X\) such that
\[
\begin{aligned}
&f_i(x_{+})>b &&\text{for }i\in P,
&
&f_i(x_{+})<a &&\text{for }i\in Q,\\
&f_i(x_{-})<a &&\text{for }i\in P,
&
&f_i(x_{-})>b &&\text{for }i\in Q.
\end{aligned}
\]
Writing
\[
S(x)=\sum_{i=1}^{m}c_i f_i(x),
\]
we obtain
\[
S(x_{+})-S(x_{-})
>
(b-a)\sum_{i=1}^{m}|c_i|.
\]
Consequently,
\[
2\left\|\sum_{i=1}^{m}c_i f_i\right\|_{\infty}
\geq
|S(x_{+})|+|S(x_{-})|
\geq
|S(x_{+})-S(x_{-})|,
\]
and therefore
\[
\left\|\sum_{i=1}^{m}c_i f_i\right\|_{\infty}
\geq
\frac{b-a}{2}\sum_{i=1}^{m}|c_i|.
\]
Thus, \((f_n)\) is an \(\ell_{1}\)-sequence.

Conversely, suppose that \((f_n)_{n\in\mathbb{N}}\subset\mathcal{F}\)
is an \(\ell_{1}\)-sequence. Thus, there exists \(c>0\) such that
\[
c\sum_{i=1}^{m}|c_i|
\leq
\left\|\sum_{i=1}^{m}c_i f_i\right\|_{\infty}
\]
for every \(m\in\mathbb{N}\) and every choice of scalars
\(c_1,\ldots,c_m\).

Since \(\mathcal{F}\) is norm bounded, there exists \(M>0\) such that
\(\|f_n\|_{\infty}\leq M\) for every \(n\). Define
\[
\Phi:X\longrightarrow [-M,M]^{\mathbb{N}},
\qquad
\Phi(x)=\bigl(f_1(x),f_2(x),\ldots\bigr),
\]
and let
\[
K=\overline{\Phi(X)}
\]
be the closure in the product topology. Then \(K\) is a compact
metrizable space.

For every \(n\in\mathbb{N}\), let
\[
p_n:K\longrightarrow \mathbb{R},
\qquad
p_n(z)=z_n,
\]
be the \(n\)-th coordinate function. Then \(p_n\in C(K)\) and
\[
p_n\circ\Phi=f_n.
\]
Moreover, for every finite sequence of scalars \(c_1,\ldots,c_m\),
the density of \(\Phi(X)\) in \(K\) gives
\[
\begin{aligned}
\left\|\sum_{i=1}^{m}c_i p_i\right\|_{C(K)}
&=
\sup_{z\in K}\left|\sum_{i=1}^{m}c_i z_i\right|\\
&=
\sup_{x\in X}\left|\sum_{i=1}^{m}c_i f_i(x)\right|\\
&=
\left\|\sum_{i=1}^{m}c_i f_i\right\|_{\infty}.
\end{aligned}
\]
Hence \((p_n)\) is an \(\ell_{1}\)-sequence in \(C(K)\).

By the standard compact-domain theorem for bounded sequences in
\(C(K)\), an \(\ell_{1}\)-sequence has an independent subsequence \cite[Theorem 3.11]{vD-89}.
Thus, after passing to a subsequence, there exist real numbers
\(a<b\) such that \((p_{n_k})_{k\in\mathbb{N}}\) is independent on
\(K\).

We claim that \((f_{n_k})_{k\in\mathbb{N}}\) is independent on \(X\)
with the same constants \(a<b\). Indeed, let \(P,Q\subset\mathbb{N}\)
be finite and disjoint. Since \((p_{n_k})\) is independent, the set
\[
U_{P,Q}
=
\bigcap_{k\in P}\{z\in K:p_{n_k}(z)<a\}
\cap
\bigcap_{k\in Q}\{z\in K:p_{n_k}(z)>b\}
\]
is a nonempty open subset of \(K\). Since \(\Phi(X)\) is dense in
\(K\), there exists \(x\in X\) such that \(\Phi(x)\in U_{P,Q}\).
It follows that
\[
f_{n_k}(x)<a\quad (k\in P),
\qquad
f_{n_k}(x)>b\quad (k\in Q).
\]
Therefore \((f_{n_k})\) is independent on \(X\), as required.
\end{proof}

\begin{cor}\label{cor:strong-tame-equals-tame}
Let \(G\) be a topological group and \(f\in \operatorname{RUC}(G)\).
Then the orbit \(fG\) contains an \(\ell_{1}\)-sequence if and only if
it contains an infinite independent sequence. 
\end{cor}

\begin{proof}
Right translations preserve the supremum norm, so
\[
\|fg\|_{\infty}=\|f\|_{\infty}
\qquad (g\in G).
\]
Thus \(fG\) is a norm-bounded subset of \(\ell_{\infty}(G)\), and
Lemma~\ref{lem:l1-independent-arbitrary-set} applies.
\end{proof}

\medskip
\noindent\textbf{Acknowledgment.}
I thank Michael Megrelishvili for helpful remarks.   ChatGPT plus was 
also helpful in locating references and in correcting errors in a
preliminary version of this note.

\end{document}